\documentclass[11pt,a4paper]{article}
\usepackage{etex} 
\usepackage[latin1]{inputenc}
\usepackage[T1]{fontenc}
\usepackage[right=2cm,left=2cm,top=2cm,bottom=2cm]{geometry}
\usepackage{ifthen}
\usepackage{enumerate}
\usepackage{enumitem}
\usepackage{amsmath}
\usepackage{amsthm}
\usepackage{amsfonts}
\usepackage{amssymb}
\usepackage{latexsym}
\usepackage{ae}
\usepackage{xcolor,colortbl} 
\usepackage{graphics,amsmath,amssymb}
\usepackage[np]{numprint}
\usepackage{hyperref}
\usepackage{multirow}
\usepackage{todonotes}

\npdecimalsign{\ensuremath{.}}

\newtheorem{theorem}{Theorem}

\newtheorem{lem}[theorem]{Lemma}

\theoremstyle{definition}

\theoremstyle{remark}

\numberwithin{equation}{section}

\makeatletter
\newcommand{\subjclass}[2][2020]{%
  \let\@oldtitle\@title%
  \gdef\@title{\@oldtitle\footnotetext{#1 \emph{Mathematics subject classification.} #2}}%
}
\newcommand{\keywords}[1]{%
  \let\@@oldtitle\@title%
  \gdef\@title{\@@oldtitle\footnotetext{\emph{Key words and phrases.} #1.}}%
}
\makeatother

\newcommand{\Z}{\mathbb {Z}}

\newcommand{\R}{\mathbb {R}}
\newcommand{\C}{\mathbb {C}}

\newcommand{\Addresses}{{% additional braces for segregating \footnotesize
  \bigskip
  \footnotesize

  O.~Bordell\`es, \textsc{2 all\'{e}e de la combe, 43000 Aiguilhe, France.}\par\nopagebreak
  \textit{E-mail address}: \texttt{borde43@wanadoo.fr}

  \medskip

  L.~T\'oth, \textsc{Department of Mathematics, University of P\'ecs, Ifj\'us\'ag \'utja 6, 7624 P\'ecs, Hungary.}\par\nopagebreak
  \textit{E-mail address}: \texttt{ltoth@gamma.ttk.pte.hu}

}}

\DeclareMathOperator{\md}{mod}

\title{{\bf Mean values of the product of an integer and its modular inverse}}
\date{}

\author{Olivier Bordell\`es and L\'aszl\'o T\'oth}



\begin{document}

\maketitle

\begin{abstract}
In this note, we extend to a composite modulo a recent result of Chan (2016) dealing with mean values of the product of an integer and its multiplicative inverse modulo a prime number. 
\end{abstract}

\section{Introduction}

\subsection{Background and motivation}

Let $k \geqslant 2$ be any fixed integer. Chan \cite{chan16} studies the sum
$$S_k(p) := \underset{a_1 \dotsb a_k \equiv 1 \; (\md p)}{\sum_{a_1=1}^{p-1} \dotsb \sum_{a_k=1}^{p-1}} a_1 \dotsb a_k$$
where $p$ is a prime number, and shows that
\begin{equation}
   S_k(p) = 2^{-k} p^{k}(p-1)^{k-1} + O \left( p^{\vartheta_k} (\log p)^k \right) \label{eq:Chan}
\end{equation}
where $\vartheta_2 = \frac{5}{2}$ and $\vartheta_k = \frac{3k}{2}$ if $k \geqslant 3$. As can be seen in \cite[Theorem~3]{chan16}, this result also holds for the more general sum
$$S_k(p,m) := \underset{a_1 \dotsb a_k \equiv m \; (\md p)}{\sum_{a_1=1}^{p-1} \dotsb \sum_{a_k=1}^{p-1}} a_1 \dotsb a_k$$
where $m$ is any positive integer satisfying $(m,p)=1$. The main aim of this work is to generalize \eqref{eq:Chan} to any composite modulo $n \geqslant 2$. We therefore investigate the sum
\begin{equation}
   S_k(n,m) := \underset{a_1 \dotsb a_k \equiv m \; (\md n)}{\sum_{a_1=1}^{n} \dotsb \sum_{a_k=1}^{n}} a_1 \dotsb a_k \label{eq:bordtoth}
\end{equation}
where $m$ is any positive integer satisfying $(m,n)=1$, and with the convention that the sum vanishes whenever there exists $j \in \{1, \dotsc, k \}$ such that $(a_j,n) > 1$. 

\subsection{Main results}

As in Chan's paper, it appears that the case $k=2$ is more intricate and needs more work. We prove the following results.

\begin{theorem}
\label{th:k_sup_3}
Let $k \geqslant 3$ be any fixed integer. Then, for $n \geqslant 2$ and $m \geqslant 1$ such that $(m,n)=1$,
$$S_k(n,m) = 2^{-k} n^k \varphi(n)^{k-1} + O^\star \left( (2 \sqrt{3} )^k n^{3k/2} (\log n)^k\right).$$ 
\end{theorem}

When $k \geqslant 4$ and $n \geqslant 3$ is odd, more recent results for mean values of odd Dirichlet characters enable us to slightly improve the above estimate.

\begin{theorem}
\label{th:k_sup_4}
For $k \geqslant 4$, $n \geqslant 3$ odd and $m \geqslant 1$ such that $(m,n)=1$,
$$S_k(n,m) = 2^{-k} n^k \varphi(n)^{k-1} + O \left( n^{\frac{3}{2}(k-2)} \varphi(n)^3 (\log n)^{k-4} (\log \log n)^2 \right).$$ 
\end{theorem}

The case $k=2$ needs specific estimates involving additive characters and exponential sums. On the other hand, as in Theorem~\ref{th:k_sup_3}, we are able to give a fully explicit error term in this case.

\begin{theorem}
\label{th:k=2}
For all $n \geqslant 2$ and $m \geqslant 1$ such that $(m,n)=1$,
$$S_2(n,m) = 4^{-1} n^2 \varphi(n) + O^\star \Biggl( n^{2} \, \sigma_{1/2}(n) \, (\tau(n) \log en)^2 + \tfrac{1}{4} n \varphi(n) \Biggr)$$ 
where $\sigma_\lambda (n)$ is the sum of the $\lambda$th powers of the divisors of $n$ and $\tau(n) = \sigma_0(n)$ is the number of divisors of $n$.
\end{theorem}

In order to study the sharpness of the previous result, it can be interesting to investigate the mean square of its error term. When $n$ is odd, we can prove the next result.

\begin{theorem}
\label{th:mean_square}
For all $n \geqslant 3$ odd and all $\varepsilon >0$ small, we have
$$\sum_{\substack{m=1 \\ (m,n)=1}}^n \left| S_2(n,m) - 4^{-1} n^2 \varphi(n) \right|^2 = \frac{5 C_n}{144} \left( n \varphi(n) \right)^3 + O \left(n^{5 + \varepsilon} \right) $$
where the constant $C_n$ is given in \eqref{eq:C} and depends only on $n$. In particular, for all $n$ odd sufficiently large, there exists an integer $m \in \{1, \dotsc,n\}$ such that $(m,n)=1$ and
$$\left| S_2(n,m) - 4^{-1} n^2 \varphi(n) \right| \gg n^{3/2} \varphi(n).$$
\end{theorem}

Note that $n^{3/2} \varphi(n) \gg n^{5/2} (\log \log n)^{-1}$ while $n^{2} \, \sigma_{1/2}(n) \, (\tau(n) \log en)^2 \ll n^{5/2+\varepsilon}$. Hence, when $n$ is odd, the error term in Theorem~\ref{th:k=2} is sharp apart from the $n^\varepsilon$-factor.

\subsection{Notation}

$n \geqslant 2$, $k \geqslant 2$ and $m \geqslant 1$ are integers, and we always assume that $(m,n)=1$. For all $x \in \R$, we set $e_n(x) := e^{\frac{2 i \pi x}{n}}$ and $\| x \|$ is the distance to the nearest integer. As for arithmetic functions, $\mu$ is the M\"{o}bius function, $\varphi$ is the Euler totient function and, for all $\lambda \in \R$, let $\sigma_\lambda (n) := \sum_{d \mid n} d^\lambda$. It is customary to set $\tau := \sigma_0$. The main sum $S_k(n,m)$ studied here is defined in \eqref{eq:bordtoth}. Finally, we will use Ramar\'{e}'s notation 
$$f(n) = M(n) + O^\star (R(n)) \iff \forall n \geqslant 2, \ \left| f(n) - M(n) \right | \leqslant R(n).$$

\section{Proofs of Theorems~\ref{th:k_sup_3} and~\ref{th:k_sup_4}}

\subsection{Theorem~\ref{th:k_sup_3}}

By the orthogonality relations of the Dirichlet characters, we derive
\begin{align*}
   S_k(n,m) &= \frac{1}{\varphi(n)} \sum_{a_1=1}^n \dotsb \sum_{a_k=1}^n a_1 \dotsb a_k \sum_{\chi \, (\md n)} \overline{\chi}(m) \chi \left( a_1 \right) \dotsb \chi \left( a_k \right) \\
   &= \frac{1}{\varphi(n)} \sum_{\chi \, (\md n)} \overline{\chi}(m) \left( \sum_{a=1}^n a \chi(a) \right)^k \\
   &=  \frac{1}{\varphi(n)} \Biggl( \sum_{\substack{a =1 \\(a,n) = 1}}^n a \Biggr)^k + \frac{1}{\varphi(n)} \sum_{\substack{\chi \, (\md n) \\ \chi \neq \chi_0}} \overline{\chi}(m) \left( \sum_{a =1}^n a \chi(a) \right)^k \\
   & = 2^{-k} n^k \varphi(n)^{k-1} + R_k(n,m)
\end{align*}
and using Abel summation and the P\'{o}lya-Vinogradov inequality, we get
\begin{equation}
   \left| \sum_{a \leqslant n} a \chi(a) \right| \leqslant 2n \, \max_{A \leqslant n} \left| \sum_{a \leqslant A} \chi(a) \right| \leqslant 2\sqrt{3} \, n^{3/2} \log n \label{eq:Abel}
\end{equation}
so that $\left| R_k(n,m) \right| \leqslant (2 \sqrt{3} )^k n^{3k/2}  \left( \log n \right)^k$.
\qed

\subsection{Theorem~\ref{th:k_sup_4}}

In this section, we assume $k \geqslant 4$ and $n \geqslant 3$ to be an \textit{odd} integer. The proof rests on the following result, which is \cite[Lemma~9]{xu08}.

\begin{lem}
\label{le:xu08}
Let $q \in \Z_{\geqslant 3}$ odd and set
\begin{equation}
   C_q := \prod_{p^\alpha \| q} \left( 1 + \frac{2p^3+p^2-1}{(p^2+1)(p^2+p+1)} - \frac{1}{p^{3 (\alpha - 1)} (p^2+p+1)} \right). \label{eq:C}
\end{equation}
Then, for all $\ell \in \Z_{\geqslant 0}$
$$\sum_{\substack{\chi \, (\md q) \\ \chi(-1) = -1}} \overline{\chi} \left( 2^\ell \right) \left| \sum_{a = 1}^{q} a \chi(a) \right|^4 = C_q \, \frac{(3 \ell +5)q^3 \varphi(q)^4}{72 \times 2^{\ell+1}} + O_\ell \left( q^{6 + \varepsilon} \right).$$
\end{lem}

It should be pointed out that, for all positive integers $q$, we have
$$\prod_{p \mid q} \left( 1 + \frac{4}{5p} \right) \leqslant C_q \leqslant \prod_{p \mid q} \left( 1 + \frac{2}{p} \right)$$
and in particular that $C_q$ is unbounded and $C_q \ll (\log \log q)^2$. We will make use of this result with $q=n$ and $\ell = 0$ in the following weaker form: under the hypotheses of the lemma, we have
\begin{equation}
   \sum_{\substack{\chi \, (\md n) \\ \chi(-1) = -1}} \left| \sum_{a = 1}^{n} a \chi(a) \right|^4 \ll n^3 \varphi(n)^4 (\log \log n)^2. \label{eq:xu08}
\end{equation}

We now are in a position to show Theorem~\ref{th:k_sup_4}. First notice that, for all $\chi \neq \chi_0$, we have the identity $\displaystyle \sum_{a=1}^{q} a \chi(a) = - \chi(-1) \sum_{a=1}^{q} a \chi(a)$ (cf. the proof of \cite[Theorem~12.20]{apo76}), so that this sum vanishes if $\chi$ is even, and hence
$$S_k(n,m) =  2^{-k} n^k \varphi(n)^{k-1} + R^{\, \textrm{odd}}_k(n,m)$$
where
$$R^{\, \textrm{odd}}_k(n,m) := \frac{1}{\varphi(n)} \sum_{\substack{\chi \, (\md n) \\ \chi(-1)=-1}} \overline{\chi}(m) \left( \sum_{a =1}^n a \chi(a) \right)^k.$$
Now using \eqref{eq:Abel} and \eqref{eq:xu08}, we derive
\begin{align*}
   \sum_{\substack{\chi \, (\md n) \\ \chi(-1)=-1}}  \left| \sum_{a =1}^n a \chi(a) \right|^k &= \sum_{\substack{\chi \, (\md n) \\ \chi(-1)=-1}} \left| \sum_{a =1}^n a \chi(a) \right|^{k-4} \left| \sum_{a =1}^n a \chi(a) \right|^4 \\
   & \ll n^{\frac{3}{2}(k-4)} (\log n)^{k-4} \times n^3 \varphi(n)^4 (\log \log n)^2 \\
   & \ll n^{\frac{3}{2}(k-2)} \varphi(n)^4(\log n)^{k-4} (\log \log n)^2
\end{align*}
implying the asserted estimate.
\qed

\section{Proof of Theorem~\ref{th:k=2}}

\subsection{Technical lemmas}

\begin{lem}
\label{le:kloos}
For all $n \in \Z_{\geqslant 1}$, $k \in \Z_{\geqslant 1}$ and $m \in \Z_{\geqslant 1}$ such that $(m,n)=1$,
$$\left| \sum_{\substack{a=1 \\ (a,n)=1}}^n a \, e_n \left( km \overline{a} \right ) \right| \leqslant 2 n^{3/2} \tau(n) (n,k)^{1/2} \log (en).$$
\end{lem}

\begin{proof}
From Weil's bound applied to Kloosterman sums and a completion device, it is known \cite[(6.21) p. 454]{bor20} that, for all positive integers $t \leqslant n$
$$\left| \sum_{\substack{a \leqslant t \\ (a,n)=1}} e_n \left( km \overline{a} \right ) \right| \leqslant n^{1/2} \tau(n) (n,mk)^{1/2} \log (en)$$
so that, by partial summation, we derive
$$\left| \sum_{\substack{a=1 \\ (a,n)=1}}^n a \, e_n \left( km \overline{a} \right ) \right|  \leqslant 2n \max_{t \leqslant n} \left| \sum_{\substack{a \leqslant t \\ (a,n)=1}}  e_n \left( km \overline{a} \right ) \right| \leqslant 2 n^{3/2} \tau(n) (n,mk)^{1/2} \log (en)$$
and we conclude the proof noticing that $(n,mk) = (n,k)$ since $(n,m)=1$.
\end{proof}

\noindent
The next lemma is well-known, but we provide here a proof for the sake of completeness.

\begin{lem}
\label{le:alk11}
Let $q \in \Z_{\geqslant 2}$ and $k \in \Z_{\geqslant 1}$. If $q \nmid k$, then
$$\sum_{b=1}^{q-1} b \, e_q (-kb) = \frac{-q}{1-e_q(-k)}.$$
\end{lem}

\begin{proof}
If $z^q = 1$, $z \ne 1$, then $\sum_{b=0}^{q-1} z^b = (1-z^q) / (1-z) = 0$. Hence
\begin{align*}
   \left( \sum_{b=1}^{q-1} bz^{b-1} \right) (1-z) &= z + 2z^2 + \dotsb + (q-1) z^{q-1} - z^2 - 2z^3 - \dotsb - (q-1)z^q \\
   &= z + z^2 + \dotsb + z^{q-1} - (q-1) z^q = -1 -(q-1) = -q.
\end{align*}
\end{proof}

\begin{lem}
\label{le:sum_tech}
For all $n \in \Z_{\geqslant 2}$ and $k \in \Z_{\geqslant 1}$,
$$\sum_{\substack{b=1 \\ (b,n)=1}}^n b  \, e_n \left( -kb \right) = -n \sum_{\substack{d \mid n \\ \frac{n}{d} \nmid k}} \frac{\mu(d)}{1-e_{n/d}(-k)} + \frac{n}{2} \sum_{\substack{d \mid n \\ d < n \\ \frac{n}{d} \mid k}} \mu(d) \left( \frac{n}{d} - 1 \right).$$
\end{lem}

\begin{proof}
We have
\begin{align*}
   \sum_{\substack{b=1 \\ (b,n)=1}}^n b \, e_n \left( -kb \right) &= \sum_{d \mid n} d \mu(d) \, \sum_{b \leqslant n/d} b \, e_n (-kdb) \\
   &= \sum_{d \mid n} d \mu(d) \, \sum_{b \leqslant n/d} b e_{n/d} (-kb) \\
   &= \sum_{d \mid n} d \mu(d) \left( \frac{n}{d} + \sum_{b \leqslant \frac{n}{d} - 1} b \, e_{n/d} (-kb) \right) \\
   &= n \underbrace{\sum_{d \mid n} \mu(d)}_{=0, \; \textrm{since} \; n \geqslant 2} + \left( \sum_{\substack{d \mid n \\ \frac{n}{d} \nmid k}} + \sum_{\substack{d \mid n \\ \frac{n}{d} \mid k}} \right) d \mu(d)\sum_{b \leqslant \frac{n}{d} - 1} b \, e_{n/d} (-kb) \\
   &= \sum_{\substack{d \mid n \\ \frac{n}{d} \nmid k}} d \mu(d) \, \frac{-n/d}{1-e_{n/d}(-k)} + \sum_{\substack{d \mid n \\ \frac{n}{d} \mid k}} d \mu(d) \sum_{b \leqslant \frac{n}{d} - 1} b \\
   &=  -n \sum_{\substack{d \mid n \\ \frac{n}{d} \nmid k}} \frac{\mu(d)}{1-e_{n/d}(-k)} + \frac{n}{2} \sum_{\substack{d \mid n \\ d < n \\ \frac{n}{d} \mid k}} \mu(d) \left( \frac{n}{d} - 1 \right)
\end{align*}
where we used Lemma~\ref{le:alk11} in the $1$st sum of the penultimate line.
\end{proof}

We will make use of the following principle.

\begin{lem}
\label{le:sum_f(n,k)}
Let $f : \Z_{\geqslant 1} \to \C$ be any arithmetic function. Then, for any $n \in \Z_{\geqslant 1}$,
$$\sum_{k=1}^n f(k) = \sum_{d \mid n} \ \sum_{\substack{h=1 \\ \left( \frac{n}{d}, h \right )=1}}^{n/d} f(hd).$$
\end{lem}

\begin{proof}
Set $d=(n,k)$ and sum over all values of $d \mid n$, assuming by convention that the inner sum vanishes if $d \nmid k$, yielding
$$\sum_{k=1}^n f(k) = \sum_{d \mid n} \ \sum_{\substack{k=1 \\ d \mid k \\ \left( \frac{n}{d}, \frac{k}{d} \right )=1}}^n f(k).$$
The change $k=hd$ then gives the asserted result.
\qedhere
\end{proof}

\begin{lem}
\label{le:sum}
Let $q \in \Z_{\geqslant 2}$ and $\ell \in \Z_{\geqslant 1}$. Then
$$\sum_{\substack{a \; (\md q) \\ q \nmid a \ell}} \frac{1}{\left| 1 - e_q(\pm \, a \ell) \right|} \leqslant \frac{q}{2}  \log \left( \frac{eq}{2(\ell,q)}\right)  < \frac{q}{2}  \log (eq).$$
\end{lem}

\begin{proof}
The inequality $|e(x) - e(y) | \geqslant 4 \| x-y \|$ \cite[Exercise~114]{bor20} and periodicity yield
$$\sum_{\substack{a \; (\md q) \\ q \nmid a \ell}} \frac{1}{\left| 1 - e_q(\pm \, a \ell) \right|} \leqslant \frac{1}{4} \sum_{\substack{a=1 \\  q \nmid a \ell}}^{q-1} \frac{1}{\| a \ell/q \|} = \frac{1}{4} \sum_{j=1}^{q-1} \frac{1}{\| j/q \|} \, \sum_{\substack{a=1 \\ a \ell \equiv j \, (\md q)}}^{q-1} 1.$$
For fixed $\ell$, $j$ and $q$, the linear congruence $a \ell \equiv j \; (\md q)$ has solutions in $a$ if and only if $(\ell, q) \mid j$ and in this case there are $(\ell, q)$ solutions $(\md q)$. Hence, with the notation $q/(\ell, q) := r$,

\begin{align*}
    \sum_{\substack{a \; (\md q) \\ q \nmid a \ell}} \frac{1}{\left| 1 - e_q(\pm \, a \ell) \right|} & \leqslant \frac{(\ell, q)}{4} \sum_{\substack{j=1 \\ (\ell, q) \mid j}}^{q-1} \frac{1}{\| j/q \|} = \frac{(\ell, q)}{4} \sum_{t=1}^{r-1} \frac{1}{\| t/r \|} \\   
   & = \frac{(\ell, q)}{4} \left( \sum_{t \leqslant r/2} \frac{r}{t} + \sum_{r/2 < t \leqslant r - 1} \frac{r}{r-t} \right)\\
   & \leqslant \frac{r(\ell, q)}{2} \sum_{t \leqslant r/2} \frac{1}{t} \leqslant \frac{q}{2} \log \left( er/2 \right) < \frac{q}{2}  \log (eq)
\end{align*}
as required.
\end{proof}

\begin{lem}
\label{le:sum_gcd(n,k)}
Let $n \in \Z_{\geqslant 2}$ and $d$ be a divisor of $n$. Then
$$\sum_{\substack{k =1 \\ \frac{n}{d} \nmid k}}^{n-1} \frac{(n,k)^{1/2}}{\left| 1 - e_{n/d}(- k) \right|} \leqslant \tfrac{1}{2} \, n^{1/2} \, \sigma_{1/2} (n) \log(en).$$
\end{lem}

\begin{proof}
Using Lemma~\ref{le:sum_f(n,k)} we first derive
$$\sum_{\substack{k =1 \\ \frac{n}{d} \nmid k}}^{n-1} \frac{(n,k)^{1/2}}{\left| 1 - e_{n/d}(- k) \right|} = \sum_{\substack{\delta \mid n \\ \delta < n}} \delta^{1/2} \, \sum_{\substack{h = 1 \\ \left( \frac{n}{\delta} , h \right) = 1 \\ \frac{n}{d} \nmid h \delta}}^{\frac{n}{\delta}-1} \frac{1}{\left| 1 - e_{n/d}(- h \delta) \right|} = \sum_{\substack{\delta \mid n \\ \delta < n}} \delta^{1/2} \, \sum_{\substack{h = 1 \\ \left( \frac{n}{\delta} , h \right) = 1 \\ \frac{n}{\delta} \nmid h d}}^{\frac{n}{\delta}-1} \frac{1}{\left| 1 - e_{n/\delta}(- h d) \right|}.$$
Notice that $\frac{n}{\delta} \geqslant 2$ since $\delta \mid n$ and $\delta < n$. Now Lemma~\ref{le:sum} with $q = \frac{n}{\delta}$ and $\ell = d$ yields
$$\sum_{\substack{k =1 \\ \frac{n}{d} \nmid k}}^{n-1} \frac{(n,k)^{1/2}}{\left| 1 - e_{n/d}(- k) \right|} \leqslant \frac{n}{2} \sum_{\delta \mid n} \delta^{-1/2} \log \left( \frac{en}{\delta}\right) \leqslant \tfrac{1}{2} \, n^{1/2} \, \sigma_{1/2} (n) \log(en).$$
\end{proof}

\subsection{Proof of Theorem~\ref{th:k=2}}

The orthogonality relations for additive characters yield
$$S_2(n,m) = \frac{1}{n} \sum_{\substack{a=1 \\ (a,n)=1}}^n a \sum_{\substack{b=1 \\ (b,n)=1}}^n b \sum_{k=1}^n e_n \left( k (m\overline{a}-b) \right) = \frac{1}{n} \sum_{k=1}^n \ \sum_{\substack{a=1 \\ (a,n)=1}}^n a \, e_n \left( km \overline{a} \right) \sum_{\substack{b=1 \\ (b,n)=1}}^n b \, e_n \left( -k b \right).$$
We start by isolating the term corresponding to $k=n$, so that
\begin{equation}
   S_2(n,m) = \tfrac{1}{4} n \left( \varphi(n) \right)^2 + \frac{1}{n} \sum_{k=1}^{n-1} \ \sum_{\substack{a=1 \\ (a,n)=1}}^n a \, e_n \left( km \overline{a} \right) \sum_{\substack{b=1 \\ (b,n)=1}}^n b \, e_n \left( -k b \right) := \tfrac{1}{4} n \left( \varphi(n) \right)^2 + R_2(n,m). \label{eq:step_1}
\end{equation}

Applying Lemma~\ref{le:sum_tech}, we first derive
\begin{align*}
   R_2(n,m) &= \frac{1}{n} \sum_{k=1}^{n-1} \ \sum_{\substack{a=1 \\ \left(a, n \right) =1}}^n a \, e_n \left( km \overline{a} \right ) \left( -n \sum_{\substack{d \mid n \\ \frac{n}{d} \nmid k}} \frac{\mu(d)}{1-e_{n/d}(-k)} + \frac{n}{2} \sum_{\substack{d \mid n \\ d < n \\ \frac{n}{d} \mid k}} \mu(d) \left( \frac{n}{d} - 1 \right) \right) \\
   &= - \sum_{d \mid n} \mu(d) \ \sum_{\substack{k=1 \\ \frac{n}{d} \nmid k}}^{n-1} \frac{1}{1-e_{n/d}(-k)} \sum_{\substack{a=1 \\ \left(a, n \right) =1}}^n a \, e_n \left( km \overline{a} \right ) + \frac{1}{2} \sum_{\substack{d \mid n \\ d < n}} \mu(d) \ \left( \frac{n}{d} - 1 \right)  \sum_{\substack{k=1 \\ \frac{n}{d} \mid k}}^{n-1} \ \sum_{\substack{a=1 \\ \left(a, n \right) =1}}^n a \, e_n \left( km \overline{a} \right ) \\
   &:= R_{21}(n,m) + R_{22}(n,m).
\end{align*}

Lemmas~\ref{le:kloos} and~\ref{le:sum_gcd(n,k)} yield
\begin{align}
   \left| R_{21}(n,m) \right| & \leqslant \sum_{d \mid n} \ \sum_{\substack{k=1 \\ \frac{n}{d} \nmid k}}^{n-1} \frac{1}{\left| 1-e_{n/d}(-k) \right|} \left|  \sum_{\substack{a=1 \\ \left(a, n \right) =1}}^n a \, e_n \left( km \overline{a} \right ) \right| \notag \\
   & \leqslant 2 n^{3/2} \tau(n) \, \log (en) \sum_{d \mid n} \ \sum_{\substack{k=1 \\ \frac{n}{d} \nmid k}}^{n-1} \frac{(n,k)^{1/2}}{\left| 1-e_{n/d}(-k) \right|} \notag \\
   & \leqslant n^{2} \, \sigma_{1/2}(n) \, \tau(n) (\log en)^2 \sum_{d \mid n} 1 = n^{2} \, \sigma_{1/2}(n) \, (\tau(n) \log en)^2. \label{eq:step_2}
\end{align}

Now setting $k= \frac{hn}{d}$ in the sum over $k$ in $R_{22}(n,m)$ yields
\begin{align*}
   R_{22}(n,m) &= \frac{1}{2} \sum_{\substack{d \mid n \\ d < n}} \mu(d) \ \left( \frac{n}{d} - 1 \right)  \sum_{h=1}^{d-1} \ \sum_{\substack{a=1 \\ \left(a, n \right) =1}}^n a \, e_d \left( hm \overline{a} \right ) \\
   &= \frac{1}{2} \sum_{\substack{d \mid n \\ 1 < d < n}} \mu(d) \ \left( \frac{n}{d} - 1 \right)  \sum_{h=1}^{d-1} \ \sum_{\substack{a=1 \\ \left(a, n \right) =1}}^n a \, e_d \left( hm \overline{a} \right ).
\end{align*}
It should be pointed out that $d \mid n$ and $d > 1$ imply that $d \nmid \overline{a}$. Now inverting the summations, we derive for all $d \mid n$ such that $1 < d < n$
\begin{align*}
   \sum_{h=1}^{d-1} \ \sum_{\substack{a=1 \\ \left(a, n \right) =1}}^n a \, e_d \left( hm \overline{a} \right ) &= \sum_{\substack{a=1 \\ \left(a, n \right) =1}}^n a \ \sum_{h=1}^{d-1} e_d \left( hm \overline{a} \right ) = \sum_{\substack{a=1 \\ \left(a, n \right) =1}}^n a \, e_d \left(m \overline{a} \right ) \times \frac{1-e_d \left( (d-1) m\overline{a} \right )}{1-e_d \left( m\overline{a} \right )} \\
   &= \sum_{\substack{a=1 \\ \left(a, n \right) =1}}^n a \, e_d \left( m\overline{a} \right ) \times \frac{1-e_d \left( - m\overline{a} \right )}{1-e_d \left(m \overline{a} \right )} = - \sum_{\substack{a=1 \\ \left(a, n \right) =1}}^n a \, e_d \left( m\overline{a} \right ) \, e_d \left( -m\overline{a} \right ) \\
   &= - \sum_{\substack{a=1 \\ \left(a, n \right) =1}}^n a = - \frac{n \varphi(n)}{2}
\end{align*}
so that, since $n \geqslant 2$
\begin{equation}
   R_{22}(n,m) = - \frac{n \varphi(n)}{4} \sum_{\substack{d \mid n \\ 1 < d < n}} \mu(d) \ \left( \frac{n}{d} - 1 \right) = \frac{n^2 \varphi(n)}{4} - \frac{n \varphi(n)^2}{4} - \frac{n \varphi(n)}{4}. \label{eq:step_3}
\end{equation}

The result follows by reporting \eqref{eq:step_2} and \eqref{eq:step_3} in \eqref{eq:step_1}. \qed

\section{Proof of Theorem~\ref{th:mean_square}}

Let $n \geqslant 3$ odd. Using orthogonality relations for Dirichlet characters, we get by the proof of Theorem~\ref{th:k_sup_3}
\begin{align*}
   & \sum_{\substack{m=1 \\ (m,n)=1}}^n \left| S_2(n,m) - 4^{-1} n^2 \varphi(n) \right|^2 \\
   & \qquad = \frac{1}{\varphi(n)^2} \sum_{\substack{m=1 \\ (m,n)=1}}^n \left| \sum_{\substack{\chi \, (\md n) \\ \chi \neq \chi_0}} \overline{\chi}(m) \left( \sum_{a=1}^n a \chi(a) \right)^2 \right|^2 \\
   & \qquad \qquad = \frac{1}{\varphi(n)^2} \sum_{\substack{\chi_1 \, (\md n) \\ \chi_1 \neq \chi_0}} \ \sum_{\substack{\chi_2 \, (\md n) \\ \chi_2 \neq \chi_0}} \left( \sum_{a=1}^n a \chi_1(a) \right)^2 \left( \sum_{a=1}^n a \overline{\chi_2} (a) \right)^2 \underbrace{\sum_{\substack{m=1 \\ (m,n)=1}}^n \chi_1(m) \, \overline{\chi_2}(m)}_{= \, \varphi(n) \ \textrm{if} \ \chi_1 = \chi_2, \ 0 \ \textrm{otherwise}} \\
   & \qquad \qquad \qquad = \frac{1}{\varphi(n)} \sum_{\substack{\chi \, (\md n) \\ \chi \neq \chi_0}} \left| \sum_{a=1}^n a \chi(a) \right|^4 = \frac{1}{\varphi(n)} \sum_{\substack{\chi \, (\md n) \\ \chi(-1) = -1}} \left| \sum_{a=1}^n a \chi(a) \right|^4
\end{align*}
and the proof follows by using Lemma~\ref{le:xu08} with $\ell = 0$.
\qed

\subsection*{Acknowledgments}

\Addresses

\end{document}